\documentclass[10pt,reqno]{amsart}
\usepackage{amsmath}
\usepackage{amsthm}
\usepackage{amssymb}
\usepackage{amsfonts}
\usepackage{latexsym}
\newcommand{\Z}{ \mathbb{Z}}
\newcommand{\R}{ \mathbb{R}}

\newcommand{\ran}{\operatorname{ran}}
\newcommand{\cran}{\overline{\operatorname{ran}}\, }

\newcommand{\norm}[1]{\left\| #1 \right\|}
\newcommand{\inner}[1]{\left< #1 \right>}

\newcommand{\N}{\mathbb{N}}

\newcommand{\h}{\mathcal{H}}
\newcommand{\K}{\mathcal{K}}
\newcommand{\minimatrix}[4]{\begin{pmatrix} #1 & #2 \\ #3 & #4 \end{pmatrix}  }

\newcommand{\twovector}[2]{\begin{pmatrix} #1\\#2 \end{pmatrix} }

\newcommand{\rank}{\operatorname{rank}}

\renewcommand{\phi}{\varphi}

\renewcommand{\S}{\mathcal{S}}

\newtheorem{Corollary}{Corollary}
\newtheorem{Theorem}{Theorem}

\newtheorem{Lemma}{Lemma}
\theoremstyle{definition}
\newtheorem*{Definition}{Definition}

\newtheorem{Example}{Example}

\newtheorem*{Question}{Question}

\allowdisplaybreaks
\begin{document}
    \title{Complex Symmetric Partial Isometries}

    \author{Stephan Ramon Garcia}
%    \author{James Tener}
    \address{   Department of Mathematics\\
            Pomona College\\
            Claremont, California\\
            91711 \\ USA}
    \email{Stephan.Garcia@pomona.edu}
    \urladdr{http://pages.pomona.edu/\textasciitilde sg064747}
    
    \author{Warren R. Wogen}
    \address{Department of Mathematics\\
	CB \#3250, Phillips Hall\\
	Chapel Hill, NC 27599}
    \email{wrw@email.unc.edu}
    \urladdr{http://www.math.unc.edu/Faculty/wrw}

    
    \thanks{First author partially supported by National Science Foundation Grant DMS-0638789.}

    \begin{abstract}
    	An operator $T \in B(\h)$ is complex symmetric if there
	exists a conjugate-linear, isometric involution $C:\h\rightarrow\h$ 
	so that $T = CT^*C$.  We provide a concrete description of all complex symmetric partial isometries.
	In particular, we prove that any partial isometry on a Hilbert space of dimension $\leq 4$
	is complex symmetric.  
    \end{abstract}

\maketitle

\section{Introduction}

The aim of this note is to complete the classification of complex symmetric partial isometries
which was started in \cite{SNCSO}.  In particular, we give a concrete necessary and sufficient
condition for a partial isometry to be a complex symmetric operator.

Before proceeding any further, let us first recall a few definitions.  In the following, $\h$
denotes a separable, complex Hilbert space and $B(\h)$ denotes the collection of all
bounded linear operators on $\h$.

\begin{Definition}
	A \emph{conjugation} is a conjugate-linear operator $C:\h \rightarrow \h$, 
	which is both \emph{involutive} (i.e., $C^2 = I$) and \emph{isometric} (i.e., $\inner{Cx,Cy} = \inner{y,x}$).
\end{Definition}

\begin{Definition}
	We say that $T \in B(\h)$ is \emph{$C$-symmetric}
	if $T = CT^*C$.  We say that $T$ is \emph{complex symmetric} if there 
	exists a conjugation $C$ with respect to which $T$ is $C$-symmetric.
\end{Definition}

It is straightforward to show that if 
$\dim \ker T \neq \dim \ker T^*$, 
then $T$ is not a complex symmetric operator.
For instance, the unilateral shift is perhaps the most ubiquitous example
of a partial isometry which is not complex symmetric 
(see \cite[Prop.~1]{CSOA}, \cite[Ex.~2.14]{CCO}, \cite[Cor.~7]{MUCFO}).
On the other hand, we have the following theorem from \cite{SNCSO}:

	\begin{Theorem}\label{TheoremPartial}
		Let $T \in B(\h)$ be a partial isometry.\smallskip
		\begin{enumerate}\addtolength{\itemsep}{0.5\baselineskip}
			\item If $\dim \ker T = \dim \ker T^* = 1$, 
				then $T$ is a complex symmetric operator,
			
			\item If $\dim \ker T \neq \dim \ker T^*$, 
				then $T$ is not a complex symmetric operator.
				
			\item If $2 \leq \dim \ker T = \dim \ker T^* \leq \infty$, then
				either possibility can (and does) occur.
		\end{enumerate}
	\end{Theorem}
	
Although these results are the sharpest possible statements that can be made
given only the data $(\dim \ker T, \dim \ker T^*)$, they are in some sense unsatisfactory.
For instance, it is known that partial isometries on $\h$ that are not complex symmetric exist if $\dim \h \geq 5$
and that every partial isometry on $\h$ is complex symmetric if $\dim \h \leq 3$,
the authors were unable to answer the corresponding question if $\dim \h = 4$.  
To be more specific, the techniques used in \cite{SNCSO} were insufficient to resolve the question in the case
where $\dim \h = 4$ and $\dim \ker T = 2$.  
Significant numerical evidence in favor of the assertion that all partial isometries
on a four-dimensional Hilbert space are complex symmetric has recently been produced 
by J.~Tener \cite{Tener}.

Suppose that $T$ is a partial isometry on $\h$ and let 
\begin{equation}\label{eq-Initial}
	\h_1 = (\ker T)^{\perp} = \ran T^*
\end{equation}
denote the \emph{initial space} of $T$ and $\h_2 = (\h_1)^{\perp} = \ker T$ denote its
orthogonal complement (see \cite[Pr.~127]{Halmos} or \cite[Ch.~VIII, Sect.~3]{Conway} for terminology).  
With respect to the orthogonal decomposition $\h = \h_1 \oplus \h_2$, we have
\begin{equation}\label{eq-Standard}
	T = \minimatrix{A}{0}{B}{0}
\end{equation}
where $A:\h_1 \rightarrow \h_1$ and $B:\h_1\rightarrow \h_2$.
Furthermore, the fact that $T^*T$ is the orthogonal projection onto $\h_1$ yields the identity
\begin{equation}\label{eq-SOS}
	A^*A + B^*B = I,
\end{equation}
where $I$ denotes the identity operator on $\h_1$.
Finally, observe that the operator $A \in B(\h_1)$ is simply 
the compression of the partial isometry $T$ to its initial space.

The main result of this note is the following concrete description of complex symmetric partial isometries:

\begin{Theorem}\label{TheoremMain}
	Let $T \in B(\h)$ be a partial isometry.  If $A$ denotes the compression of 
	$T$ to its initial space, then $T$ 
	is a complex symmetric operator if and only if $A$ is a complex symmetric operator.
\end{Theorem}

Due to its somewhat lengthy and computational proof, we defer the proof
of the preceding theorem until Section \ref{SectionProof}.
We remark that Theorem \ref{TheoremMain} remains true if one instead considers
the final space of $T$.  Indeed, simply apply the theorem with $T^*$ in place
of $T$ and then take adjoints.

\begin{Corollary}
	Every partial isometry of rank $\leq 2$ is complex symmetric.
\end{Corollary}

\begin{proof}
	Let $T \in B(\h)$ be a partial isometry such that $\rank T \leq 2$.  If $\rank T = 0$, then $T = 0$
	and there is nothing to prove.
	If $\rank T = 1$, then this is handled in \cite{SNCSO}.
	In the case $\rank T = 2$, we may write
	\begin{equation*}
		T = \minimatrix{A}{0}{B}{0}
	\end{equation*}
	where $A$ is an operator on a two-dimensional space.  Since every operator on a two-dimensional
	Hilbert space is complex symmetric (see \cite[Cor.~3]{Balayan}, \cite[Cor.~3.3]{Chevrot}, 
	\cite[Ex.~6]{CSOA}, \cite[Cor.~1]{SNCSO}, \cite[Cor.~3]{Tener}), 
	the desired conclusion follows from Theorem \ref{TheoremMain}. 
\end{proof}

\begin{Corollary}
	Every partial isometry on a Hilbert space of dimension $\leq 4$ is complex symmetric.
\end{Corollary}

\begin{proof}
	As mentioned earlier, the results of \cite{SNCSO} indicate that 
	only the case $\dim \h = 4$ and $\dim \ker T = 2$ requires resolution.
	The corollary is now immediate consequence of Theorem \ref{TheoremMain} and the fact
	that every operator on a two-dimensional Hilbert space is complex symmetric.
\end{proof}

We conclude this section with the following theorem, which asserts that each $C$-symmetric
partial isometry can be extended to a $C$-symmetric unitary operator on the whole space (the 
significance lies in the fact that the corresponding conjugations for these two operators are the same).

\begin{Theorem}
	If $T$ is a $C$-symmetric partial isometry, then there exists
	a $C$-symmetric unitary operator $U$ and an orthogonal projection $P$ such that $T = UP$.
\end{Theorem}

\begin{proof}
	Since $T$ is a $C$-symmetric partial isometry, it follows that $|T| = P$ is an orthogonal projection
	and that $T = CJP$ where $J$ is a conjugation supported on $\ran P$ which commutes with $P$ 
	\cite[Sect.~2.2]{CSO2}.
	We may extend $J$ to a conjugation $\widetilde{J}$ on all of $\h$ by forming the internal direct sum
	$J \oplus J'$ where $J'$ is a partial conjugation supported on $\ker P$.  The operator $U = C\widetilde{J}$
	is a $C$-symmetric unitary operator.
\end{proof}

%%%%%%%%%%%%%%%%%%%%%%%%%%%%%%%%%%
\section{Partial isometries and the norm closure problem}

	Partial isometries on infinite-dimensional spaces often provide
	examples of note.  For instance, one can give
	a simple example of a partial isometry $T$ 
	satisfying $\dim \ker T = \dim \ker T^* = \infty$ which is
	not a complex symmetric operator:

	\begin{Example}\label{ExampleDirectSum}
		Let $S$ denote the unilateral shift on $l^2(\N)$,
		Although $S$ is certainly \emph{not} a complex symmetric operator
		(by (ii) of Theorem \ref{TheoremPartial}, see also \cite[Ex.~2.14]{CCO}, or \cite[Cor.~7]{MUCFO}), 
		part (i) of Theorem \ref{TheoremPartial} does ensure that the partial isometry 
		$S \oplus S^*$ \emph{is} complex symmetric.
		Indeed, simply take $N$ to be the bilateral shift on $l^2(\Z)$ and note that $S \oplus S^*$ is unitarily equivalent
		to $N - Ne_0 \otimes e_0$.  	
		That $S \oplus S^*$ is complex symmetric can also be verified by a 
		direct computation \cite[Ex.~5]{CSO2}.
		On the other hand, the partial isometry $T = S \oplus 0$ on $l^2(\N) \oplus l^2(\N)$
		is \emph{not} a complex symmetric operator by Lemma \ref{LemmaZero}.
	\end{Example}
	
	Let $\S(\h)$ denote the subset of $B(\h)$ consisting of all bounded 
	complex symmetric operators on $\h$.
	There are several ways to think about $\S(\h)$.  By definition, we have
	\begin{equation*}
		\S(\h) = \{ T \in B(\h) :  \text{$\exists$ a conjugation $C$ s.t.~$T = CT^*C$} \}.
	\end{equation*}
	If $C$ is a fixed conjugation on $\h$, then we also have
	\begin{equation*}
		\S(\h) = \{  UTU^* :  T = CT^*C,\,\, \text{$U$ unitary} \}.
	\end{equation*}
	Thus if we identify $\h$ with $l^2(\N)$ and $C$ denotes the canonical conjugation on $l^2(\N)$
	(i.e., entry-by-entry complex conjugation),
	we can think of $\S(\h)$ as being the \emph{unitary orbit} of the set of all bounded (infinite) complex symmetric matrices.

	The following example shows that the set $\S(\h)$ is not closed in the strong operator topology (SOT):

	\begin{Example}\label{ExampleLast}
		We maintain the notation of Example \ref{ExampleDirectSum}.
		For $n \in \N$, let $P_n$ denote the orthogonal projection onto
		the span of the basis vectors $\{e_i : i \geq n\}$ of $l^2(\N)$.
		Now observe that each operator $T_n = P_n S \oplus S^*$ is unitarily equivalent to
		$S \oplus 0_n \oplus S^*$ where $0_n$ denotes the zero operator on an $n$-dimensional Hilbert space.
		Each $T_n$ is complex symmetric since $S \oplus S^*$ is complex symmetric (by Lemma \ref{LemmaZero}).
		On the other hand, since $P_n S$ is SOT-convergent to $0$, it follows that
		the SOT-limit of the sequence $T_n$ is $0 \oplus S^*$, which is not a complex symmetric operator
		(by Lemma \ref{LemmaZero}).
	\end{Example}

	The preceding example demonstrates 
	that the set of all complex symmetric operators (on a fixed, infinite-dimensional Hilbert
	space $\h$) is not SOT-closed.  We also remark that 
	the conjugations corresponding to the operators $T_n$ from 
	Example \ref{ExampleLast} depend on $n$.
	In contrast, if we fix a conjugation $C$, then it is elementary to see that the set of $C$-symmetric operators is a 
	SOT-closed subspace of $B(\h)$.	

	We conclude with a related question, which we have been unable to resolve:

	\begin{Question}
		Is $\S(\h)$ norm closed?
	\end{Question}

%%%%%%%%%%%%%%%%%%%%%%%%%%%%%%%%%%
\section{Proof of Theorem \ref{TheoremMain}}\label{SectionProof}

	This entire section is devoted to the proof of Theorem \ref{TheoremMain}.
	We first require the following lemma:

	\begin{Lemma}\label{LemmaZero}
		If $\h,\K$ are separable complex Hilbert spaces, then
		$T \in B(\h)$ is a complex symmetric operator if and only if 
		$T \oplus 0 \in B(\h \oplus \K)$ is a complex symmetric operator.
	\end{Lemma}

	\begin{proof}
		If $T$ is a $C$-symmetric operator on $\h$, then it is easily verified that
		$T \oplus 0$ is $(C \oplus J)$-symmetric on $\h \oplus \K$ for any conjugation $J$ on $\K$.
		The other direction is slightly more difficult to prove.
		
		Suppose that $S = T\oplus 0$ is a complex symmetric operator on $\h \oplus \K$.
		Before proceeding any further, let us remark that it suffices to consider the case where 
		\begin{equation}\label{eq-Density}
			\h = \overline{\ran T + \ran T^*}.
		\end{equation}
		Otherwise let $\h_1 = \overline{\ran T + \ran T^*}$ and note that $\h_1$ is a reducing subspace of $\h$.
		If $\h_2$ denotes the orthogonal complement of $\h_1$ in $\h$, then with respect
		to the orthogonal decomposition $\h_1 \oplus \h_2 \oplus \K$, 
		the operator $S$ has the
		form $T' \oplus 0 \oplus 0$, where $T'$ denotes the restriction of $T'$ to $\h_1$.
		By now considering $S$ with respect to the orthogonal decomposition 
		$\h \oplus \K = \h_1 \oplus (\h_2 \oplus \K)$, it follows that we need only consider the case 
		where \eqref{eq-Density} holds.
		
		Suppose now that \eqref{eq-Density} holds and that
		$S$ is $C$-symmetric where $C$ denotes a conjugation
		on $\h \oplus \K$.  Writing the
		equations $CS = S^*C$ and $CS^* = SC$ in terms of the 
		$2 \times 2$ block matrices
		\begin{equation}\label{eq-SC}
			S = \minimatrix{T}{0}{0}{0}, \qquad C = \minimatrix{C_{11}}{C_{12}}{C_{21}}{C_{22}}
		\end{equation}
		(the entries $C_{ij}$ of $C$ are conjugate-linear operators), we find that
		\begin{align}
			C_{11}T &= T^*C_{11} \label{eq-Vanish01},\\
			C_{21} T &= C_{21}T^* = 0, \label{eq-Vanish03}\\
			T^*C_{12} &= TC_{12} = 0\label{eq-Vanish04}.
		\end{align}

		Since $C_{21}T = C_{21}T^* = 0$, it follows that $C_{21}$ vanishes 
		on $\ran T + \ran T^*$ and hence on $\h$ itself by \eqref{eq-Density}.
		On the other hand, \eqref{eq-Vanish04} implies that $C_{12}$
		vanishes on the orthogonal complements of $\ker T$ and $\ker T^*$ in $\h$.
		By \eqref{eq-Density}, this implies that $C_{12}$ vanishes identically.
		
		It follows immediately from \eqref{eq-SC} that $C_{11}$ and $C_{22}$ must be
		conjugations on $\h$ and $\K$, respectively, whence $T$ is $C_{11}$-symmetric
		by \eqref{eq-Vanish01}.  This concludes the proof of the lemma.
	\end{proof}

	Now let us suppose that $T$ is a partial isometry on $\h$ and let 
	\begin{equation*}
		\h_1 = (\ker T)^{\perp} = \ran T^*.
	\end{equation*}
	and $\h_2 = \ker T$.
	With respect to the decomposition $\h = \h_1 \oplus \h_2$, it follows that
	\begin{equation*}
		T = \minimatrix{A}{0}{B}{0}
	\end{equation*}
	where $A:\h_1 \rightarrow \h_1$, $B:\h_1\rightarrow \h_2$, and
	\begin{equation}\label{eq-SOS}
		A^*A + B^*B = I.
	\end{equation}
	
	\noindent $(\Rightarrow)$
	Suppose that $T$ is a complex symmetric operator.
	For an operator with polar decomposition $T = U|T|$ (i.e., $U$ is the unique partial isometry
	satisfying $\ker U = \ker T$ and $|T|$ denotes the positive operator $\sqrt{T^*T}$), the 
	\emph{Aluthge transform} of $T$ is defined to be the operator $\widetilde{T} = |T|^{ \frac{1}{2} } U |T|^{ \frac{1}{2} }$.
	Noting that 
	\begin{equation*}
		T^*T = \minimatrix{I}{0}{0}{0},
	\end{equation*}
	we find that 
	\begin{equation*}
		\widetilde{T} = \minimatrix{A}{0}{0}{0}.
	\end{equation*}
	By \cite[Thm.~1]{ATCSO}, we know that the Aluthge transform of a complex symmetric operator
	is complex symmetric.  Applying Lemma \ref{LemmaZero} to $\widetilde{T}$, we conclude that $A$ 
	is complex symmetric, as desired.
	\medskip

	\noindent $(\Leftarrow)$
	Let us now consider the more difficult implication of Theorem \ref{TheoremMain}, 
	namely that if $A$ is a complex symmetric operator, then $T$ is as well.
	We claim that it suffices to consider the case where $\cran B = \h_2$.  
	In other words, we argue that if 
	\begin{equation*}
		\K = \overline{ \ran T + \ran T^*} ,
	\end{equation*}
	then we may suppose that $\K = \h$.  Indeed, $\K$ is a reducing subspace for $T$ 
	and $T = 0$ on $\K^{\perp}$.  By Lemma \ref{LemmaZero}, if $T|_{\K}$ is a complex
	symmetric operator, then so is $T$.	
%	In other words, we argue that we may suppose that
%	Indeed, otherwise the closure of $\ran T + \ran T^*$ would be a proper subspace
%	of $\h$ and $T$ would have a $0$ direct summand since
%	\begin{equation*}
%		(\ran T + \ran T^*)^{\perp} = \ker T \cap \ker T^*.
%	\end{equation*}
%	By Lemma \ref{LemmaZero}, it follows that we need only consider the 
%	case where $\cran B = \h_2$.  

	Write $B = V|B|$ where $V:\h_1 \rightarrow \h_2$ is a partial isometry with initial space
	$(\ker B)^{\perp} \subseteq \h_1$ and final space $\h_2$ (since $\cran B = \h_2$).
	In particular, we have the relations
	\begin{equation}\label{eq-W1}
		V^*B = |B| = B^*V, \qquad |B| = \sqrt{I - A^*A}.
	\end{equation}
	By hypothesis, the operator $A \in B(\h_1)$ is complex symmetric.
	Therefore suppose that $K$ is a conjugation on $\h_1$ such that $KA = A^*K$ and 
	observe that the equations
	\begin{align*}
		A \sqrt{I - A^*A} &= \sqrt{I -AA^*}A,\\
		A^* \sqrt{I - AA^*} &= \sqrt{I - A^*A} A^*,\\
		K \sqrt{I - A^*A} &= \sqrt{I - AA^*}K,\\
		K \sqrt{I - AA^*} &= \sqrt{I - A^*A}K,
	\end{align*}
	follow from a standard polynomial approximation argument (i.e., if $p(x) \in \R[x]$,
	then $Ap(A^*A) = p(AA^*)A$ and $Kp(A^*A) = p(AA^*)K$ hold whence the desired identities 
	follow upon passage to the strong operator limit).
	In particular, it follows from the preceding that
	\begin{equation*}
		(KA)\sqrt{I - A^*A} = \sqrt{I - A^*A}(KA),
	\end{equation*}
	that is
	\begin{equation}\label{eq-W2}
		KA|B| = |B|KA, \qquad
		A^*K|B| = |B|A^*K.
	\end{equation}

	Let us now define a conjugate-linear operator $C$ on $\h$
	by the formula
	\begin{equation}\label{eq-C}
		C = \minimatrix{AK}{KB^*}{BK}{-VA^*KV^*}.
	\end{equation}
	Assuming for the moment that $C$ is a conjugation on $\h$,
	we observe that
	\begin{equation*}
		\underbrace{ \minimatrix{A}{0}{ B }{0}	}_T
		=
		\underbrace{ \minimatrix{AK}{KB^*}{BK}{-VA^*KV^*}  }_C
		\underbrace{ \minimatrix{K}{0}{0}{0} }_{J}
		\underbrace{ \minimatrix{I}{0}{0}{0} }_{|T|}.
	\end{equation*}
	Since it is clear that $J$ is a partial conjugation which is supported on the range of $|T|$ and which commutes with $|T|$,
	it follows immediately that $T$ is a $C$-symmetric operator (see \cite[Thm.~2]{CSO2}).

	To complete the proof of Theorem \ref{TheoremMain}, we must therefore show that $C$ is a conjugation on $\h$.
	In other words, we must check that $C^2$ is the identity operator on $\h$ and that $C$ is isometric.
	Since these computations are somewhat lengthy, we perform them separately:\medskip

	\noindent\textbf{Claim}:  $C^2 = I$.

	\begin{proof}[Pf.~of Claim]
		We first expand out $C^2$ as a $2 \times 2$ block matrix:
		\begin{align*}	
			C^2
			&=\minimatrix{AK}{KB^*}{BK}{-VA^*KV^*}\minimatrix{AK}{KB^*}{BK}{-VA^*KV^*}\\
			&= \minimatrix{AKAK + KB^*BK}{AKKB^*-KB^*VA^*KV^*}{BKAK-VA^*KV^*BK}{BKKB^*+VA^*KV^*VA^*KV^*}\\		
			&= \minimatrix{AA^* + KB^*BK}{AB^*-KB^*VA^*KV^*}{BA^*-VA^*KV^*BK}{BB^*+VA^*KV^*VA^*KV^*}.	
		\end{align*}
		To obtain the preceding line, we used the fact that $K$ is a conjugation and $A$ is $K$-symmetric.
		Letting $E_{ij}$ denote the entries of the preceding block matrix we find that
		\begin{align*}
			E_{11} 
			&= AA^* + KB^*BK &&\\
			&= AA^* + K(I - A^*A)K &&\\
			&= AA^* + (I - AA^*) &&\\
			&= I.&&\\
			&&&\\
			E_{12}
			&= AB^*-KB^*VA^*KV^*&&\\
			&= AB^* - K|B|A^*KV^*   &&\text{by \eqref{eq-W1}}\\
			&= AB^* - KA^*K|B|V^* &&\text{by \eqref{eq-W2}}\\
			&= AB^* - A|B|V^* &&\\
			&= AB^* - AB^* &&\text{since $B^* = |B|V$}\\
			&= 0.&&\\
			&&&\\
			E_{21}
			&= BA^*-VA^*KV^*BK &&\\
			&= BA^* - VA^*K|B|K &&\text{since $V^*B = |B|$}\\
			&= BA^* - V|B|A^*KK && \text{by \eqref{eq-W2}}\\
			&= BA^* - V|B|A^* && \\
			&= BA^* - BA^* && \text{since $B = V|B|$}\\
			&= 0.&&
		\end{align*}
		As for $E_{22}$, it suffices to show that $E_{22}$ agrees with $I$ (the identity operator on $\h_2$) on the
		range of $B$, which is dense in $\h_2$.  In other words, we wish to show that
		$E_{22}Bx = Bx$ for all $x \in \h_2$, which is equivalent to showing that
		\begin{equation}\label{eq-Mess}
			E_{22}Bx = BB^*Bx+VA^*KV^*VA^*KV^*Bx = Bx
		\end{equation}
		for all $x \in \h_2$.
		Let us investigate the second term of \eqref{eq-Mess}:
		\begin{align*}
			VA^*KV^*VA^*KV^*Bx
			&= VA^*KV^*VA^*K|B|x &&\text{by \eqref{eq-W1}}\\
			&= VA^*KV^*V|B|A^*Kx &&\text{by \eqref{eq-W2}}\\
			&= VA^*K|B|A^*Kx &&\text{since $V^*V = P_{\cran|B|}$}\\
			&= V|B|A^*KA^*Kx &&\text{by \eqref{eq-W2}}\\
			&= BA^*KA^*Kx &&\text{since $B = V|B|$}\\
			&= BA^*Ax && \\
			&= B(I - B^*B)x &&\text{since $A^*A+B^*B = I$}\\
			&= Bx - BB^*Bx. &&
		\end{align*}
		Putting this together with \eqref{eq-Mess}, we find that $E_{22}Bx = Bx$ for all $x \in \h_2$
		whence $E_{22} = I$, as claimed.
	\end{proof}

	\noindent\textbf{Claim}:  $C$ is isometric.

	\begin{proof}[Pf.~of Claim]
		The proof requires three steps:\smallskip
		\begin{enumerate}\addtolength{\itemsep}{0.5\baselineskip}
			\item Show that $C$ is isometric on $\h_1$,
			\item Show that $C$ is isometric on $B\h_1$, which is dense in $\h_2$,
			\item Show that $C\h_1 \perp C(B\h_1)$.
		\end{enumerate}
		\smallskip
		For the first portion, observe that
		\begin{align*}
			\norm{ C \twovector{x}{0} }^2
			&= \norm{ \minimatrix{AK}{KB^*}{BK}{-VA^*KV^*}  \twovector{x}{0} }^2\\
			&= \norm{ \twovector{AKx}{BKx} }^2\\
			&= \inner{AKx,AKx} + \inner{BKx,BKx}\\
			&= \inner{A^*AKx,Kx} + \inner{B^*BKx,Kx}\\
			&= \inner{(A^*A + B^*B)Kx,Kx}\\
			&= \inner{Kx,Kx}\\
			&= \norm{Kx}^2\\
			&= \norm{x}^2.
		\end{align*}
		Thus (i) holds.
		\medskip

		\noindent Now for (ii):
		\begin{align*}
			\norm{ C \twovector{0}{Bx} }^2
			&= \norm{ \minimatrix{AK}{KB^*}{BK}{-VA^*KV^*}  \twovector{0}{Bx} }^2\\
			&= \norm{  \twovector{KB^*Bx}{-VA^*KV^*Bx} }^2\\
			&= \norm{KB^*Bx}^2 + \norm{VA^*KV^*Bx}^2\\
			&= \norm{B^*Bx}^2 + \norm{VA^*K|B|x}^2\\
			&= \norm{B^*Bx}^2 + \norm{V|B|A^*Kx}^2\\
			&= \norm{B^*Bx}^2 + \norm{BA^*Kx}^2\\
			&= \norm{B^*Bx}^2 + \inner{BA^*Kx,BA^*Kx}\\
			&= \norm{B^*Bx}^2  + \inner{B^*BA^*Kx,A^*Kx}\\
			&= \norm{B^*Bx}^2  + \inner{(I - A^*A)A^*Kx,A^*Kx}\\
			&= \norm{B^*Bx}^2  + \inner{A^*K(I - A^*A)x,A^*Kx}\\
			&= \norm{B^*Bx}^2  + \inner{K(I - A^*A)x,AA^*Kx}\\
			&= \inner{B^*Bx,B^*Bx}  + \inner{KAA^*Kx,(I - A^*A)x}\\
			&= \inner{(I - A^*A)x,(I - A^*A)x} + \inner{A^*Ax,(I - A^*A)x}\\
			&= \inner{x,(I - A^*A)x}  - \inner{A^*Ax,(I - A^*A)x}+ \inner{A^*Ax,(I - A^*A)x}\\
			&= \inner{x,(I - A^*A)x}  \\
			&= \inner{x,B^*Bx}  \\
			&= \inner{Bx,Bx}  \\
			&= \norm{Bx}^2.
		\end{align*}
		Thus (ii) holds.
		\medskip

		\noindent Now for (iii):
		\begin{align*}
			\inner{ C \twovector{x}{0}, C \twovector{0}{By} }
			&= \inner{\minimatrix{AK}{KB^*}{BK}{-VA^*KV^*}\twovector{x}{0}, 
				\minimatrix{AK}{KB^*}{BK}{-VA^*KV^*} \twovector{0}{By} }\\
			&= \inner{ \twovector{AKx}{BKx}, \twovector{KB^*By}{-VA^*KV^*By} }\\
			&= \inner{ AKx, KB^*By} - \inner{BKx, VA^*KV^*By}\\
			&= \inner{B^*By, KAKx } - \inner{BKx, VA^*K|B|y}\\
			&= \inner{B^*By, A^*x } - \inner{BKx, V|B|A^*Ky}\\
			&= \inner{AB^*By, x } - \inner{BKx, BA^*Ky}\\
			&= \inner{AB^*By, x } - \inner{B^*BKx, A^*Ky}\\
			&= \inner{AB^*By, x } - \inner{(I - A^*A)Kx, A^*Ky}\\
			&= \inner{AB^*By, x } - \inner{K(I - AA^*)x, A^*Ky}\\
			&= \inner{AB^*By, x } - \inner{KA^*Ky, (I - AA^*)x}\\
			&= \inner{AB^*By, x } - \inner{Ay, (I - AA^*)x}\\
			&= \inner{AB^*By, x } - \inner{(I - AA^*)Ay, x}\\
			&= \inner{AB^*By, x } - \inner{A(I - A^*A)y, x}\\
			&= \inner{AB^*By, x } - \inner{AB^*By, x}\\
			&=0.
		\end{align*}
		By the polarization identity, it follows that
		\begin{equation*}
			\inner{ C\twovector{x_1}{Bx_2}, C \twovector{y_1}{By_2} }
			= \inner{ \twovector{x_2}{By_2}, \twovector{x_1}{By_1} }
		\end{equation*}
		holds for all $x_1,x_2,y_1,y_2\in \h_1$ whence $C$ is isometric on $\h$.
	\end{proof}

\end{document}